\documentclass[11pt]{article}

\usepackage{amssymb, amsmath, fullpage, amsthm}
\usepackage{graphics,graphicx} 

\newtheorem{theorem}{Theorem}[section]
\newtheorem{corollary}[theorem]{Corollary}
\newtheorem{lemma}[theorem]{Lemma}
\newtheorem{proposition}[theorem]{Proposition}

\newtheorem{claim}[theorem]{Claim}

\theoremstyle{definition}
\newtheorem{question}[theorem]{Question}
\newtheorem{remark}[theorem]{Remark}
\newtheorem{example}[theorem]{Example}

\numberwithin{equation}{section}

\newcommand{\Rrr}{\mathbb R}
\newcommand{\Zzz}{\mathbb Z}
\DeclareMathOperator{\Inv}{Inv}
\DeclareMathOperator{\head}{head}
\DeclareMathOperator{\tail}{tail}

\newcommand{\SSSS}{{\mathfrak S}}
\newcommand{\AS}{\widetilde{\mathfrak S}}

\newcommand{\ffrac}[2]{#1/#2}

\newcommand\CB{\ensuremath{\mathit{CB}}}

\newcommand{\mB}{{\mathcal B}}
\newcommand{\mP}{{\mathcal P}}

\begin{document}

\title{Number of cycles in the graph
of $312$-avoiding permutations}
\author{Richard Ehrenborg,
            Sergey Kitaev\thanks{The corresponding author: phone: +44141 5483473; fax: +44141 5484523; email: sergey.kitaev@cis.strath.ac.uk}\ \ 
            and
            Einar Steingr\'{\i}msson}
\date{} 
\maketitle
\thispagestyle{empty}

\begin{abstract}
The graph of overlapping permutations is defined in a way analogous
to the De Bruijn graph on strings of symbols.
That is, for every permutation $\pi = \pi_{1} \pi_{2} \cdots \pi_{n+1}$
there is a directed edge from
the standardization of
$\pi_{1} \pi_{2} \cdots \pi_{n}$
to the standardization of
$\pi_{2} \pi_{3} \cdots \pi_{n+1}$.
We give a formula for the number of cycles of length $d$
in the subgraph of overlapping $312$-avoiding permutations.
Using this we also give a refinement of the enumeration of
$312$-avoiding affine permutations
and point out some open problems on this graph,
which so far has been little studied.
\end{abstract}

\section{Introduction and preliminaries}

One of the classical objects in combinatorics is the {\em De Bruijn graph}.  This is the directed graph on vertex set $\{0,1, \ldots, q-1\}^{n}$, the set of all strings of length $n$ over an alphabet of size~$q$, whose directed edges go from each vertex $x_{1} \cdots x_{n}$ to each vertex $x_{2} \cdots x_{n+1}$.  That is, there is a directed edge from
the string $\mathbf{x}$ to
the string $\mathbf{y}$ if and only if the last $n-1$ coordinates of $\mathbf{x}$ agree with the first $n-1$ coordinates of $\mathbf{y}$.

The De Bruijn graph has been much studied, especially in connection with combinatorics on words, and one of its well known
properties
(see for instance~\cite[page~126]{Golomb})
is the fact that its number of directed cycles of length~$d$, for $d \leq n$, is given by
\begin{equation}
  \frac{1}{d} \sum_{e | d} \mu\left(\ffrac{d}{e}\right) q^{e},
  \label{equation_de_Bruijn}
\end{equation}
where the sum is over all divisors $e$ of length $d$, and where $\mu$ denotes the number theoretic {\em M\"obius function}.  Recall that $\mu(n)$ is $(-1)^{k}$ if $n$ is a product of $k$ distinct primes and is zero otherwise.  

A natural variation on the De Bruijn graphs is obtained by replacing words over an alphabet by permutations of the set of integers $\{1,2,\ldots,n\}$, where the overlapping condition determining directed edges in De Bruijn graphs is replaced by the condition that the head and tail of two permutations have the same \emph{standardization}, that is, that their letters appear in the same order of size.
This {\em graph of overlapping permutations}, denoted $G(n)$,
has a directed edge for each permutation
$\pi \in \SSSS_{n+1}$
from the standardization of $\pi_{1} \pi_{2} \cdots \pi_{n}$
to the standardization of $\pi_{2} \pi_{3} \cdots \pi_{n+1}$.
As an example,
there is a directed edge
from $2341$ to $3412$ in $G(4)$
labeled $24513$, since
the standardizations
of $2451$ and $4513$
are $2341$ and $3412$, respectively.
In fact, between these two vertices
there is another directed edge
labeled $34512$.
The simple case of $n=2$ is illustrated in Figure~\ref{figure_1}.  Note that, apart from the path and cycle graphs mentioned in Section~\ref{section_compositions}, all graphs in this paper are directed, although we do not  explicitly refer to them as directed graphs.

The graph $G(n)$ appeared in~\cite{CDG} in connection with
{\em universal cycles on permutations}.  It has also appeared
in~\cite{Ehrenborg_Kitaev_Perry}, where it was used as a tool in determining the asymptotic behavior of consecutive pattern avoidance, and in \cite{AK}, where it is called the {\em graph of overlapping patterns} (see also \cite[Section 5.6]{Kitaev}).

What is the number of directed cycles in this graph $G(n)$,
that is,
the analogue to the question for which
Equation~\eqref{equation_de_Bruijn} is the answer.
This is a natural question which does
not seem to have been studied so far.
We have not been able to solve that problem (and we do not recognize the associated number sequences).  We do here, however, solve that problem when the graph is restricted to permutations of length~$n$ avoiding the pattern 312, that is, permutations containing no three letters the first of which is the largest and the second of which is the smallest.
We prove in Theorem~\ref{theorem_number_of_cycles}
that the number of directed cycles of length~$d$
in the restriction of the graph
is
\begin{equation}
  \frac{1}{d} \sum_{{e|d}} \mu\left(\ffrac{d}{e}\right)\binom{2e}{e},
  \label{main_formula}
\end{equation}
for $d$ not exceeding $n$.  
Note the similarity between this and the expression
in Equation~\eqref{equation_de_Bruijn}: the power $q^{e}$
in~\eqref{equation_de_Bruijn} is replaced here by
the central binomial coefficient $\binom{2e}{e}$.

It is easy to see, due to straightforward symmetries, that permutations avoiding a particular one of the patterns 132, 213 and 231 yield a graph isomorphic to the one for 312, which is the representative we have chosen.  It is also easy to see that permutations avoiding 123 (or, equivalently, 321) give rise to a nonisomorphic graph.  For this latter case we have no solution for the number of cycles, and we do not recognize the number sequences counting the cycles in that graph.

Using similar techniques, we prove that the number of $312$-avoiding
affine permutations in $\AS_{d}$ with $k$ cut points
is given by the binomial coefficient~$\binom{2d-k-1}{d-1}$.
This refines a result of Crites \cite{Crites} who showed that the number of $312$-avoiding affine permutations is $\binom{2d-1}{d-1}$.  As a corollary to our results we show that each affine permutation has a cut point  or is, in other words, decomposable.

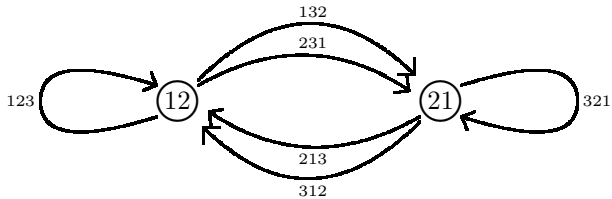
\begin{figure}
\setlength{\unitlength}{0.7 mm}
\begin{center}
  \begin{picture}(100,40)(-50,-20) \thicklines

\newcommand{\vertexcirclesmall}{\circle{7.5}}
\put(-25,0){\vertexcirclesmall}
\put(25,0){\vertexcirclesmall}
\put(-2.5,-1.4){
\put(-25,0){\small $12$}
\put(25,0){\small $21$}
}

\qbezier(-21,4)(0,25)(20,5)
\put(20,5){\qbezier(0,0)(0,1.5)(0,3)
               \qbezier(0,0)(-1.5,0)(-3,0)}
\put(-2,16){\tiny $132$}

\qbezier(-21,3)(0,15)(19,2)
\put(19,2){\qbezier(0,0)(-0.2,1.32)(-0.4,2.64)
                 \qbezier(0,0)(-1.32,-0.2)(-2.64,-0.4)}
\put(-2,10){\tiny $231$}

\qbezier(21,-3)(0,-15)(-19,-2)
\put(-19,-2){\qbezier(0,0)(0.2,-1.32)(0.4,-2.64)
                 \qbezier(0,0)(1.32,0.2)(2.64,0.4)}
\put(-2,-12){\tiny $213$}

\qbezier(21,-4)(0,-25)(-20,-5)
\put(-20,-5){\qbezier(0,0)(0,-1.5)(0,-3)
               \qbezier(0,0)(1.5,0)(3,0)}
\put(-2,-18){\tiny $312$}

\qbezier(-29,-3)(-51,-10)(-51,0)
\qbezier(-29,3)(-51,10)(-51,0)
\put(-57.5,-1){\tiny $123$}
\put(-29,3){\qbezier(0,0)(-0.5,1.2)(-1.0,2.4)
                  \qbezier(0,0)(-1.2,-0.5)(-2.4,-1.0)}

\qbezier(29,3)(51,10)(51,0)
\qbezier(29,-3)(51,-10)(51,0)
\put(52,-1){\tiny $321$}
\put(29,-3){\qbezier(0,0)(0.5,-1.2)(1.0,-2.4)
                  \qbezier(0,0)(1.2,0.5)(2.4,1.0)}
              \end{picture}
            \end{center}
            \caption{The graph $G(2)$ of overlapping permutations.}
            \label{figure_1}
          \end{figure}

The connection between cycles in the graph of overlapping permutations and affine permutations goes through bi-infinite sequences.  A bi-infinite sequence $(\ldots, f(-1), f(0), f(1), f(2), \ldots)$ of distinct real numbers yields a bi-infinite walk where the $i$th edge is given by the standardization of $f=(f(i),f(i+1), \ldots, f(i+n))$, that is, the unique permutation of $\{1,2,\ldots,n+1\}$ whose letters appear in the same order of size as the numbers in $f$ .  This walk is a closed walk of length~$d$ if the sequence is periodic, in the sense that $f(i) < f(j)$ if and only if $f(i+d) < f(j+d)$.  Thus, infinite $d$-periodic sequences correspond to $d$-cycles.


The paper is organized as follows.  In Section \ref{sec2} we introduce some  definitions related to pattern avoidance, affine permutations and bi-infinite sequences.  The last of these play an important role in the proof of the main result, as do ordinary and cyclic compositions of an integer, which are introduced in Section~\ref{section_compositions}.  In Section~\ref{321-affine-perms} we give results on the number of affine 312-avoiding permutations with a given number of cut points and show that every such permutation does have a cut point.  In Section~\ref{sec-graph} we present the main result, about the number of $d$-cycles in $G(n,312)$ and subsequently give a bijection that proves this, in Sections~\ref{sec-bijection} and~\ref{sec-inv-bijection}.  Finally, in Section~\ref{sec-open}, we list several open problems in this area.

\section{Pattern-avoiding permutations, affine permutations
and bi-infinite sequences}\label{sec2}

We first introduce some formal definitions that are needed later on. 

For a permutation $x = x_{1} \cdots x_{n}$ consisting of distinct real numbers, let $\Pi(x)$ denote the {\em standardization of $x$}, also known as the {\em reduced form of $x$}, that is, the unique permutation $\pi=\pi_1\cdots\pi_n$ in the symmetric group $\SSSS_{n}$ whose elements have the same relative order as those in $x$. In other words, $x_{i} < x_{j}$ if and only if $\pi_{i} < \pi_{j}$ for all $1 \leq i < j \leq n$ and $\pi$ is built on the set $\{1,2,\ldots,n\}$. For example, $\Pi(3(-2)02)=4123$.

The graph of overlapping permutations $G(n)$ has the elements of the symmetric group $\SSSS_{n}$ as its vertex set and for every permutation $\sigma = \sigma_{1} \cdots \sigma_{n+1}$ in $\SSSS_{n+1}$ there is a directed edge from $\Pi(\sigma_{1} \cdots \sigma_{n})$ to $\Pi(\sigma_{2}\cdots \sigma_{n+1})$.

A permutation $\pi=\pi_1\pi_2\cdots\pi_n \in \SSSS_{n}$ {\em avoids} a permutation $\tau \in \SSSS_{k}$ if there are no integers $1 \leq i_{1} < i_{2} < \cdots < i_{k} \leq n$ such that $\Pi(\pi_{i_{1}}\pi_{i_{2}} \cdots \pi_{i_{n}}) = \tau$. In this context, $\tau$ is called a {\em pattern} and we say that $\pi$ avoids the pattern $\tau$.  Let $\SSSS_{n}(\tau)$ denote the set of $\tau$-avoiding permutations in $\SSSS_{n}$.  Especially, we are interested in $312$-avoiding permutations, which are those that have no indices $i < j < k$ such that $\pi_j < \pi_k < \pi_i$.  It is well known that the number of $312$-avoiding permutations in $\SSSS_{n}$ is given by the $n$th {\em Catalan number} $C_{n}=\frac{1}{n+1} \binom{2n}{n}$.

A {\em cut point} of a permutation $\pi \in \SSSS_{n}$ is an index $j$ with $1 \leq j \leq n-1$ such that for all $i$ and $k$ satisfying $1 \leq i \leq j < k \leq n$ we have that $\pi_i < \pi_k$.  The cut points split a permutation into \emph{components}, each ending at a cut point.  A permutation without cut points is said to be {\em indecomposable} (or, sometimes, {\em irreducible}).  As an example, the permutation 31246758 has three cut points, namely 3, 4, and 7, and components 312, 4, 675 and 8, wheres 2413 is indecomposable.  The following proposition is well known
(see, for instance, \cite{Claesson_Kitaev}),
and it is easy to prove, e.g.\ using the following argument.  Every $312$-avoiding permutation is of the form $A1B$, where each element of $B$ is larger than every element of $A$, which implies that a $312$-avoiding permutation is indecomposable precisely when $B$ is empty.  Such permutations with $B$ empty directly correspond bijectively to the 312-avoiding 
permutations $A$ of length $n-1$, which leads to the following result.
\begin{proposition}
The number of $312$-avoiding indecomposable permutations in $\SSSS_{n}$
is given by the Catalan number $C_{n-1}$.
\label{proposition_indecomposable}
\end{proposition}

One extension of the notion of a permutation
is the notion of an affine permutations.
While the symmetric group~$\SSSS_{d}$ is the Weyl group~$A_{d-1}$, the group of affine permutations~$\AS_{d}$ is the affine Weyl group~$\widetilde{A}_{d-1}$.  However, the combinatorial description of affine permutations is due to Lusztig (unpublished) and the first combinatorial study of them was conducted in~\cite{Bjorner_Brenti,Ehrenborg_Readdy}.  An {\em affine permutation}
$\pi \in \AS_{d}$
is a bijection $\pi : \Zzz \longrightarrow \Zzz$ such that
\begin{align}
& \pi(i+d) = \pi(i) + d, \label{equation_affine_permutation_i} \\
& \sum_{i=0}^{d-1} (\pi(i) - i) = 0.
\end{align}
Note that the first condition implies that the values $\pi(0)$ through $\pi(d-1)$ determine the whole affine permutation.

We now extend the notion of an affine permutation to bi-infinite sequences.  A {\em bi-infinite sequence} is defined to be an injective function $f : \Zzz \longrightarrow \Rrr$.  Alternatively, one can think of a bi-infinite sequence as a bi-infinite list $(\ldots, f(-1), f(0), f(1), f(2), \ldots)$ of distinct real numbers.  We say that two bi-infinite sequences $f$ and $g$ are equivalent if there is a strictly increasing continuous function $T : \Rrr \longrightarrow \Rrr$ such that $g(i) = T(f(i))$.
Equivalently, it is enough to
assume that $T$
is strictly increasing and bijective.
It is straightforward that this relation
is an equivalence relation.  We think about the equivalence classes as bi-infinite permutations.  Hence, it is natural to extend notions from permutation patterns theory to bi-infinite sequences.

A {\em cut point} for a bi-infinite sequence $f$ is an
index $j$ such that for all integers $i \leq j < k$ we have that
$f(i) < f(k)$.
The {\em inversion set} for a bi-infinite sequence $f$ is the set
$$  \Inv(f) = \{(i,j) \in \Zzz^{2} \:\: : \:\: i < j, f(i) > f(j) \} . $$
A bi-infinite sequence is {\em periodic} with period $d$, if for all integers $i$ and $j$, the inequality $f(i) < f(j)$ is equivalent to the inequality $f(i+d) < f(j+d)$.  Equivalently, a bi-infinite sequence is periodic with period $d$ if the inversion set satisfies the condition $(i,j) \in \Inv(f)$ is equivalent to $(i+d, j+d) \in \Inv(f)$.  Extending the notion of pattern avoidance, we say that a bi-infinite sequence $f$ avoids the pattern $\sigma \in \SSSS_{n}$ if there are no integers $i_{1} < i_{2} < \cdots < i_{n}$ such that $\Pi(f(i_{1})f(i_{2}) \cdots f(i_{n})) = \sigma$.

\section{Compositions and cyclic compositions}
\label{section_compositions}

A composition of a non-negative integer $d$ into $k$ parts is a list of $k$ positive integers $(a_{1}, a_{2}, \ldots, a_{k})$ such that their sum is $d$.  Let $\alpha_{1}, \alpha_{2}, \ldots$ be a sequence of numbers and $f(t) = \sum_{i \geq 1} \alpha_{i} t^{i}$ be the associated generating function.  Form a new sequence $(\beta_{d,k})_{d \geq 1}$ by the relation
$$  \beta_{d,k}
= \sum_{(a_{1}, a_{2}, \ldots, a_{k})} \alpha_{a_{1}} \alpha_{a_{2}} \cdots \alpha_{a_{k}}
$$
where the sum is over all compositions of $d$ into $k$ parts.  Additionally, we set $\beta_{0,k}$ to be the {\em Kronecker delta} $\delta_{0,k}$, which is equal to 1 if $k=0$ and $0$ otherwise.  Also, let the sequence $(\beta_{d})_{d\geq 0}$ be defined by the sum $\beta_{d} = \sum_{k \geq 0} \beta_{d,k}$.  The following relations are classical generatingfunctionology:

\begin{equation}
       \sum_{d \geq 0} \beta_{d,k} t^{d}
   =
       (f(t))^{k}   
\:\:\:\: \text{ and } \:\:\:\:       
       \sum_{d \geq 0} \beta_{d}  t^{d}
   =
       \frac{1}{1 - f(t)}.
       \label{equation_beta_d_k}
     \end{equation}

For a family of sets $S_{i}$,
where $\alpha_{i}$ is the cardinality of a set $S_{i}$, we can give a combinatorial interpretation to the number~$\beta_{d,k}$, hence also $\beta_{d}$.  An {\em enriched composition} is a pair $({\bf a}, {\bf s})$ where ${\bf a}$ is a composition 
$(a_{1}, a_{2}, \ldots, a_{k})$ of~$d$ into $k$ parts and ${\bf s} = (s_{1}, s_{2}, \ldots, s_{k})$ is a list of the same length such that the element $s_{i}$ belongs to the set~$S_{a_{i}}$.  Now $\beta_{d,k}$ is the number of enriched compositions of $d$ into $k$ parts, and $\beta_{d}$ is the number of enriched compositions of $d$.

Note that each connected subgraph of a path is also a path.
Hence a composition 
$(a_{1}, a_{2}, \ldots, a_{k})$
of $d$ can be thought of as a subgraph of the path on $d$ vertices,
where $a_{i}$ is the size of the $i$th connected component.
The number of connected components of the subgraph is the number of parts of the composition.  With this analogue in mind we define a {\em cyclic composition} to be a subgraph of the labeled
cycle on $d$ vertices
where each component is a path.  Note that we rule out the case of the cycle being a subgraph of itself.  Yet again, the number of paths is the number of components $k$.  Observe that $k$ is also the number of edges removed to obtain the subgraph.  Since there are $d$ edges in a cycle, we have $\binom{d}{k}$ cyclic compositions of $d$ into $k$ parts for $k \geq 1$.  For instance, there are $\binom{4}{2} = 6$ cyclic compositions of $4$ into two parts, namely, two consisting of two $2$s and four consisting of $1$ and $3$ (see Figure~\ref{figure_two}).

\begin{figure}[ht]
\setlength{\unitlength}{0.7 mm}
\begin{center}
  \begin{picture}(145,15)(0,0)

\multiput(0,0)(27,0){6}
{
\put(0,0){\circle*{2}}
\put(10,0){\circle*{2}}
\put(0,10){\circle*{2}}
\put(10,10){\circle*{2}}
\put(-3,11){\small 1}
\put(11,11){\small 2}
\put(11,-4){\small 3}
\put(-3,-4){\small 4}
}

\newcommand{\north}{\put(0,10){\line(1,0){10}}}
\newcommand{\east}{\put(10,0){\line(0,1){10}}}
\newcommand{\south}{\put(0,0){\line(1,0){10}}}
\newcommand{\west}{\put(0,0){\line(0,1){10}}}

\thicklines

\put(0,0){\north \east}

\put(27,0){\north \west}

\put(54,0){\south \west}

\put(81,0){\south \east}

\put(108,0){\north \south}

\put(135,0){\east \west}

\end{picture}
\end{center}
\caption{The $6$ cyclic compositions of $4$ into two parts.}
\label{figure_two}
\end{figure}
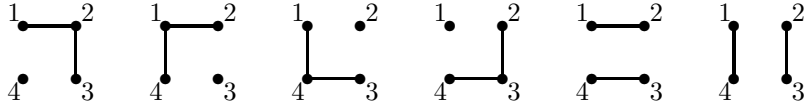

To be more formal, let $\Zzz_{d}$ denote both the integers modulo $d$ and the cycle of length $d$, where we connect $i$ and $i+1$ modulo $d$.  A cyclic composition is a set partition
$\mP = \{\mB_{1}, \mB_{2}, \ldots, \mB_{k}\}$
where each block $\mB_{i}$ is a path in the cycle $\Zzz_{d}$.  Equivalently, each block $\mB_{i}$ is the image of an interval~$[p_{i},q_{i}]$ of integers under the quotient map $\Zzz \longrightarrow \Zzz_{d}$ with the restriction $0 \leq q_{i} - p_{i} \leq d-1$.  Also, let $a_{i}$ be $q_{i} - p_{i} + 1$, that is, the cardinality of the interval $[p_{i},q_{i}]$ and its
associated path.

Similar
to compositions,
we construct new sequences $(\gamma_{d,k})_{d \geq 1}$
and $(\gamma_{d})_{d \geq 1}$ as follows:
$$  \gamma_{d,k}
= \sum_{\mP} \alpha_{a_{1}} \alpha_{a_{2}} \cdots \alpha_{a_{k}}
$$
where the sum is over all cyclic compositions $\mP$ of $d$ into $k$ parts and $a_{i}$ is the size of the $i$th part.  Also, let $\gamma_{d}$ denote the sum $\gamma_{d} = \sum_{k \geq 1} \gamma_{d,k}$.

\begin{proposition}
The generating functions for 
$\gamma_{d,k}$ and $\gamma_{d}$ are given by
\begin{align}
       \sum_{d \geq k} \gamma_{d,k} t^{d}
   & =
       t  f^{\prime}(t) (f(t))^{k-1}
   =
       \frac{t}{k}  D\left( (f(t))^{k} \right)  \mbox{ and }
       \label{equation_gamma_d_k} \\
       \sum_{d \geq 1} \gamma_{d} t^{d} & = \frac{t f^{\prime}(t)}{1 - f(t)},
       \label{equation_gamma_d}
     \end{align}
where $D$ is the differential operator with respect to $t$.
\end{proposition}
\begin{proof}
To observe the first relation, consider the component containing the vertex $1$ of the cycle.  Also, assume that this component has size $i$.  Then there are $i$ possibilities how to choose this component.  This is encoded by the generating function $\sum_{i \geq 1} i \alpha_{i} t^{i} = t f^{\prime}(t)$.  Next we have to choose a composition of $d-i$ into $k-1$ parts, which is given by $(f(t))^{k-1}$. The first result follows from multiplication of generating functions.  The second result follows from summing
Equation~\eqref{equation_gamma_d_k} over all~$k$.
\end{proof}

As a brief example of Equation~\eqref{equation_gamma_d}, note that setting $\alpha_{i} = 1$ enumerates the number of cyclic compositions. We have $f(t) = 1/(1-t) - 1$ and obtain $\sum_{d \geq 1} \gamma_{d} t^{d} = 1/(1-2t) - 1/(1-t)$, yielding the answer of $2^{d} - 1$ for the number of cyclic compositions of $d$.

Combining generating functions in Equations~\eqref{equation_beta_d_k}
and~\eqref{equation_gamma_d_k} we have the following result.
\begin{corollary}
The two quantities 
$\beta_{d,k}$ and $\gamma_{d,k}$
are related by
\begin{equation}
\gamma_{d,k} = \frac{d}{k}  \beta_{d,k}.
\label{equation_beta_gamma}
\end{equation}
\label{corollary_beta_gamma}
\end{corollary}

An {\em enriched cyclic composition} is a pair $(\mP, {\bf s})$ where $\mP$ is a cyclic composition $\{\mB_{1}, \mB_{2}, \ldots, \mB_{k}\}$ of~$d$ into $k$ parts and ${\bf s} = (s_{1}, s_{2}, \ldots, s_{k})$ is a list of length $k$ such that the element $s_{i}$ belongs to the set~$S_{|\mB_{i}|}$.  Now $\gamma_{d,k}$ has the combinatorial interpretation as the number of enriched cyclic compositions of $d$ into $k$ parts, and $\beta_{d}$ is the number of enriched cyclic compositions of $d$.

Let $C(t)$ and $\CB(t)$ denote the generating functions for the Catalan numbers and the central binomial coefficients, that is,
\begin{align*}
  C(t) & = \sum_{d \geq 0} C_{d}  t^{d} = \frac{1 - \sqrt{1 - 4t}}{2t} , \\
  \CB(t) & = \sum_{d \geq 0} \binom{2d}{d} t^{d} = \frac{1}{\sqrt{1 - 4t}} .
\end{align*}

\begin{lemma}
  Let $\alpha_{i} = C_{i-1} + \delta_{i,1}$ where $\delta_{i,1}$ denotes the Kronecker delta.  Then the central binomial coefficient is given by the sum
$$  
\binom{2d}{d} = \sum_{P} \alpha_{a_{1}} \alpha_{a_{2}} \cdots \alpha_{a_{k}},
$$ 
where the sum is over all cyclic compositions $P$ of $d$ and $d \geq 1$.
\label{lemma_central_binomial}
\end{lemma}
\begin{proof}
  First, observe that $\sum_{i \geq 1} \alpha_{i} t^{i} = t + t C(t)$.  The result now follows from Equation~\eqref{equation_gamma_d} and the identity
\begin{align*}
\CB(t) - 1
& = \frac{t (t + t C(t))^{\prime}}{1 - t - t C(t)} .  
\qedhere
\end{align*}
\end{proof}

The following two well-known identities involving the Catalan numbers are worth keeping in mind in the next section, where we prove similar results regarding affine $312$-avoiding permutations.  We have
\begin{equation}
      \sum_{(a_{1}, a_{2}, \ldots, a_{k})}
                C_{a_{1}-1} C_{a_{2}-1} \cdots C_{a_{k}-1}
    =
        \frac{k}{d} \binom{2d-k-1}{d-1}  ,
        \label{equation_A009766} 
      \end{equation}   
where the sum is over all compositions of $d$ into $k$ parts.  The numbers in the right hand side give Catalan's triangle, sequence A009766 in \cite{oeis}.  One of many things they enumerate is the set of $312$-avoiding permutations of length $d$ that split into (at most) $k$ components.  Namely, such a permutation can be decomposed as $A_{1} A_{2} \cdots A_{k}$ where every letter of $A_{i}$ is smaller than each letter of $A_{j}$ for $i<j$.  Since each component is an indecomposable $312$-avoiding permutation, these permutations with $k$ components are counted by the left hand side.  By a similar argument we have
\begin{equation*}
  \sum_{(a_{1}, a_{2}, \ldots, a_{k})}
  C_{a_{1}-1} C_{a_{2}-1} \cdots C_{a_{k}-1}
  =
  C_{d} , 
\end{equation*}   
where the sum is over all compositions of $d$.

By combining Corollary~\ref{corollary_beta_gamma} and
Equation~\eqref{equation_A009766} we have:
\begin{lemma}
The following identity holds
\begin{equation}
      \sum_{\mP}
       C_{a_{1}-1} C_{a_{2}-1} \cdots C_{a_{k}-1}
    =
        \binom{2d-k-1}{d-1} ,
        \label{equation_cyclic_composition_Catalan}
      \end{equation}
where the sum is over all cyclic compositions $\mP$
of $d$ into $k$ parts.
\label{lemma_cyclic_composition_Catalan}
\end{lemma}

\section{$312$-avoiding affine permutations}\label{321-affine-perms}

Before we continue, we take a detour to study $312$-avoiding affine permutations.  Recall that an affine permutation $\pi \in \AS_{d}$ is $312$-avoiding if there are no integers $i < j < k$ such that $\pi(j) < \pi(k) < \pi(i)$.  Crites~\cite{Crites} proved the following result for affine permutations.
\begin{theorem}[Theorem~6 in~\cite{Crites}]\label{thm-crites}
  The number of $312$-avoiding affine permutations in $\AS_{d}$ is given by~$\binom{2d-1}{d}$.
\end{theorem}
We give a refinement of this result.
Recall that a cut point for an affine permutation $\pi$
is an index $j$ such that for $i \leq j < k$ the inequality $\pi(i) < \pi(k)$ holds.
Note that, if $j$ is a cut point,
then so is any index congruent to $j$ modulo $d$.
Hence, we count the number of equivalence classes of cut points.

\begin{theorem}
  Let $k$ be a positive integer and $k\leq d$.  The number of $312$-avoiding affine permutations in $\AS_{d}$ that have $k$ cut points modulo $d$ is given by
$$ \binom{2d-k-1}{d-1} . $$
\label{theorem_k_cuts_312}
\end{theorem}
\begin{proof}
Consider a cyclic composition $\{\mB_{1}, \mB_{2}, \ldots, \mB_{k}\}$ into $k$ parts of the cycle $\Zzz_{d}$ and enrich each block with an indecomposable $312$-avoiding permutation.
By Proposition~\ref{proposition_indecomposable}  
there are $C_{a-1}$ indecomposable $312$-avoiding permutations
in $\SSSS_{a}$,
hence the total number of enriched cyclic compositions is given by Lemma~\ref{lemma_cyclic_composition_Catalan}, that is, $\binom{2d-k-1}{d-1}$.  Let the $i$th block $\mB_{i}$ be the image of the interval~$[p_{i},q_{i}]$.  View the permutation $\pi_{i}$ enriching $\mB_{i}$ as a bijection on this interval. That is, we have the bijection $\pi_{i} : [p_{i},q_{i}] \longrightarrow [p_{i},q_{i}]$.  Now concatenate these $k$ bijections, that is, define
$$ 
\pi : \bigcup_{1 \leq i \leq k} [p_{i},q_{i}] = [p_{1},q_{k}] \longrightarrow [p_{1},q_{k}]
$$ 
by $\pi(j) = \pi_{i}(j)$ if $j \in [p_{i},q_{i}]$.  Finally, extend $\pi$ to all integers by Condition~\eqref{equation_affine_permutation_i}.  By construction, it is clear that $\pi$ is a $312$-avoiding affine permutation with $k$ cut points and that all $312$-avoiding affine permutations with $k$ cut points are constructed this way.
\end{proof}

\begin{corollary}
  Every $312$-avoiding affine permutation in $\AS_{d}$ has a cut point.
  \label{corollary_there_is_a_cut point}
\end{corollary}
\begin{proof}
  Since the sum of $\binom{2d-k-1}{d-1}$ for $k$ from $1$ to $d$ is $\binom{2d-1}{d}$, and by Theorem~\ref{thm-crites}, all the $312$-avoiding affine permutations have been accounted for.
\end{proof}

A more direct proof is as follows.
\begin{proof} [A second proof of Corollary~\ref{corollary_there_is_a_cut point}.]
Let $\pi$ be an affine 312-avoiding permutation in $\AS_{d}$.
Let $P$ be the set $\{(i,\pi(i)) \: : \: i \in \Zzz\}$.
Define a poset structure on the set $P$
by $(i,\pi(i)) <_{P} (j,\pi(j))$ if $i > j$ and $\pi(i) < \pi(j)$.
Note, for instance, that the set $\{(i,\pi(i)) \: : \: i \equiv j \bmod d\}$ forms an infinite antichain.

  Now if an element $(i,\pi(i))$ is greater than both of the elements $(j,\pi(j))$ and $(k,\pi(k))$, and these two elements are incomparable, then this triple forms a $312$-pattern.  Since we assumed $\pi$ is $312$-avoiding we have that the lower order ideal generated by a single element is a chain. In other words, the poset is a forest with the minimal elements as roots.

  Pick a minimal element $(i,\pi(i))$ in the poset.  We claim that $i$ is a cut point.  Among the maximal elements above $(i, \pi(i))$ in the poset pick an element $(j,\pi(j))$ with the largest second coordinate.  We claim that if $k \leq i$ then $\pi(k) \leq \pi(j)$.  There is nothing to prove if $\pi(k) \leq \pi(i)$.  If $\pi(k) \geq \pi(i)$ then we have $(k,\pi(k)) \geq_{P} (i,\pi(i))$.  But we picked $(j,\pi(j))$ to be the element with largest second coordinate, which proves the claim.  Next we claim that if $k > i$ then $\pi(k) > \pi(j)$.  Assume that $\pi(k) < \pi(j)$. Since $(i,\pi(i))$ is a minimal element we know that $\pi(i) < \pi(j)$.  However, this yields a contradiction since $\Pi(\pi(j), \pi(i), \pi(k)) = 312$, the prohibited pattern. Hence, $i$ is a cut point.
\end{proof}

\section{The graph of $312$-avoiding permutations}\label{sec-graph}

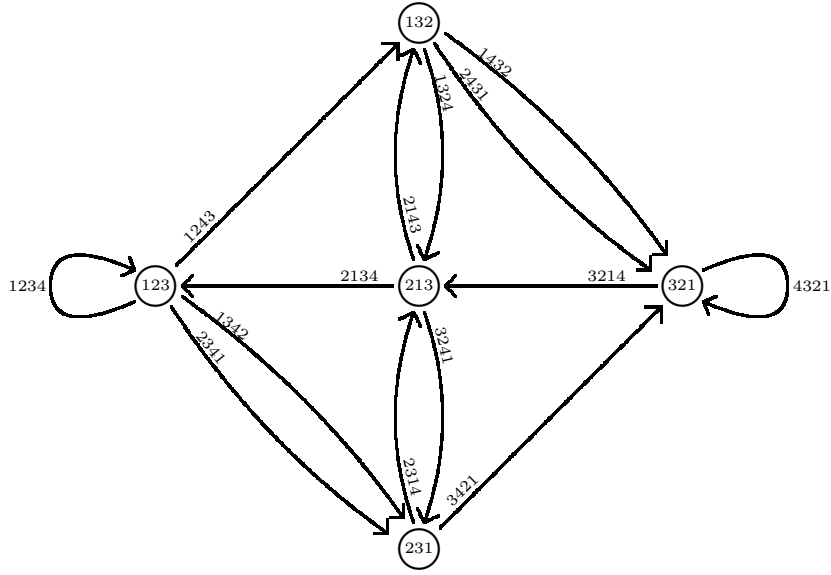
\begin{figure}
\setlength{\unitlength}{0.7 mm}
\begin{center}
  \begin{picture}(100,100)(-50,-50) \thicklines

\newcommand{\vertexcircle}{\circle{7.5}}
\put(0,0){\vertexcircle}
\put(50,0){\vertexcircle}
\put(0,50){\vertexcircle}
\put(-50,0){\vertexcircle}
\put(0,-50){\vertexcircle}
\put(-2.75,-0.8){
\put(0,0){\tiny $213$}
\put(50,0){\tiny $321$}
\put(0,50){\tiny $132$}
\put(-50,0){\tiny $123$}
\put(0,-50){\tiny $231$}
}

\qbezier(-45,0)(-25,0)(-5,0)
\qbezier(-45,0)(-44,1)(-43,2)
\qbezier(-45,0)(-44,-1)(-43,-2)
\put(-15,1){\tiny $2134$}

\qbezier(45,0)(25,0)(5,0)
\qbezier(5,0)(6,1)(7,2)
\qbezier(5,0)(6,-1)(7,-2)
\put(32,1){\tiny $3214$}

\qbezier(1,-45)(8,-25)(1,-5)
\put(1,-45){\qbezier(0,0)(1.3,0.65)(2.7,1.3)
                  \qbezier(0,0)(-0.65,1.3)(-1.3,2.7)}
\put(3,-8){\rotatebox{-77}{\tiny $3241$}}

\qbezier(-1,-45)(-8,-25)(-1,-5)
\put(-1,-5){\qbezier(0,0)(-1.3,-0.65)(-2.7,-1.3)
                 \qbezier(0,0)(0.65,-1.3)(1.3,-2.7)}
\put(-3,-33){\rotatebox{-77}{\tiny $2314$}}

\qbezier(1,5)(8,25)(1,45)
\put(1,5){\qbezier(0,0)(1.3,0.65)(2.7,1.3)
               \qbezier(0,0)(-0.65,1.3)(-1.3,2.7)}
\put(3,40){\rotatebox{-77}{\tiny $1324$}}

\qbezier(-1,5)(-8,25)(-1,45)
\put(-1,45){\qbezier(0,0)(-1.3,-0.65)(-2.7,-1.3)
                  \qbezier(0,0)(0.65,-1.3)(1.3,-2.7)}
\put(-3,17){\rotatebox{-77}{\tiny $2143$}}

\qbezier(-46,4)(-25,25)(-4,46)
\qbezier(-4,46)(-4,46)(-7,46)
\qbezier(-4,46)(-4,46)(-4,43)
\put(-45,8){\rotatebox{45}{\tiny $1243$}}

\qbezier(46,-4)(25,-25)(4,-46)
\qbezier(46,-4)(46,-4)(43,-4)
\qbezier(46,-4)(46,-4)(46,-7)
\put(5,-42){\rotatebox{45}{\tiny $3421$}}

\qbezier(47,6)(30,30)(5,48)
\put(47,6){\qbezier(0,0)(-1.25,0)(-2.5,0)
                 \qbezier(0,0)(0,1.25)(0,2.5)}
\put(11,44){\rotatebox{-38}{\tiny $1432$}}

\qbezier(44,3)(20,20)(3,46)
\put(44,3){\qbezier(0,0)(-1.25,0)(-2.5,0)
                 \qbezier(0,0)(0,1.25)(0,2.5)}
\put(7.5,40.5){\rotatebox{-51}{\tiny $2431$}}

\qbezier(-3,-44)(-20,-20)(-45,-2)
\put(-3,-44){\qbezier(0,0)(-1.25,0)(-2.5,0)
                   \qbezier(0,0)(0,1.25)(0,2.5)}
\put(-39,-6){\rotatebox{-38}{\tiny $1342$}}

\qbezier(-6,-47)(-30,-30)(-47,-4)
\put(-6,-47){\qbezier(0,0)(-1.25,0)(-2.5,0)
                 \qbezier(0,0)(0,1.25)(0,2.5)}
\put(-42.5,-9.5){\rotatebox{-51}{\tiny $2341$}}

\qbezier(-54,-3)(-70,-10)(-70,0)
\qbezier(-54,3)(-70,10)(-70,0)
\put(-78,-1){\tiny $1234$}
\put(-54,3){\qbezier(0,0)(-0.5,1.2)(-1.0,2.4)
                  \qbezier(0,0)(-1.2,-0.5)(-2.4,-1.0)}

\qbezier(54,3)(70,10)(70,0)
\qbezier(54,-3)(70,-10)(70,0)
\put(71,-1){\tiny $4321$}
\put(54,-3){\qbezier(0,0)(0.5,-1.2)(1.0,-2.4)
                  \qbezier(0,0)(1.2,0.5)(2.4,1.0)}

\end{picture}
\end{center}
\caption{The graph $G(3, 312)$, which has two $1$-cycles, two $2$-cycles and six $3$-cycles.}
\label{figure_three}
\end{figure}

Recall that $G(n)$ denotes the directed graph of overlapping permutations, that is, it has the vertex set $\SSSS_{n}$ and for every permutation $\sigma=\sigma_1\cdots\sigma_{n+1}$ in $\SSSS_{n+1}$ there is a directed edge from $\Pi(\sigma_{1}\cdots \sigma_{n})$ to $\Pi(\sigma_{2}\cdots \sigma_{n+1})$ labelled $\sigma$.  Furthermore, let $G(n,\tau)$ be the graph of overlapping $\tau$-avoiding permutations, that is, it is the subgraph of $G(n)$ having the vertex set $\SSSS_{n}(\tau)$ and the edge set~$\SSSS_{n+1}(\tau)$.  For an example, see~Figure~\ref{figure_three} where the graph $G(3,312)$ is presented.

A {\em closed walk} of length $d$ in a graph is a list of $d$ edges $(e_{1}, e_{2}, \ldots, e_{d})$ such that $\head(e_{i}) = \tail(e_{i+1})$ for $1 \leq i \leq d-1$ and $\head(e_{d}) = \tail(e_{1})$, where for a directed edge $e$, $\head(e)$ is the node the edge points to, while $\tail(e)$ is the other node incident to $e$.  Thus, $(1342, 2314, 2134)$ and its cyclic shift $(2134, 1342, 2314)$ are two different closed walks.

Define an equivalence on the set of closed walks by cyclic shifting, that is, two closed walks $(e_{1}, e_{2}, \ldots, e_{d})$ and $(e_{i}, e_{i+1}, \ldots, e_{d}, e_{1}, e_{2}, \ldots, e_{i-1})$ are equivalent.  Then a {\em $d$-cycle} is defined to be an equivalence class of size $d$.  For instance, the graph $G(3,312)$ in Figure~\ref{figure_three} has six closed walks of length $2$, namely,
$$   
(1234,1234), (4321,4321), (1324,2143), (2143,1324), (2314,3241) \mbox{ and } (3241,2314) .
$$
However, the graph $G(3,312)$ has only two $2$-cycles, since the first two closed walks yield $1$-cycles while the fourth (respectively, sixth) walk is equivalent to the third (respectively, fifth) walk.

The number of closed walks of a fixed length
is given by the following result.

\begin{theorem}
  The number of closed walks of length $d$ in $G(n,312)$, for $d \leq n$, is given by~$\binom{2d}{d}$.
  \label{theorem_closed_walks}
\end{theorem}
A bijective proof of Theorem~\ref{theorem_closed_walks} will be given in the following two sections.

\begin{theorem}
  The number of $d$-cycles in $G(n,312)$, for $d \leq n$, is given by
$$   
\frac{1}{d} \sum_{e|d} \mu\left(\ffrac{d}{e}\right) \binom{2e}{e}.
$$
\label{theorem_number_of_cycles}
\end{theorem}
\begin{proof}
  Let $h(d)$ denote the number of $d$-cycles.  A closed walk of length $d$ can be obtained by choosing a divisor $e$ of $d$, an $e$-cycle and a starting point on the cycle. By repeating the $e$-cycle $d/e$ times we obtain a closed walk of length $d$.  Hence, we have
$$   \binom{2d}{d} 
= \sum_{e|d} e \cdot h(e) . $$ The result now follows by classical M\"obius inversion.
\end{proof}

\section{The bijection}\label{sec-bijection}

It remains to show that the number of closed walks of length $d$ in $G(n,312)$ is given by $\binom{2d}{d}$.  We do this by constructing a bijection between closed walks and enriched cyclic compositions.
The proof of this bijection involves bi-infinite sequences.
Let $W_{d}$ denote the set of all closed walks of length~$d$ in the graph $G(n,312)$.

Given a cyclic composition $\mP$ on $\Zzz_{d}$, we enrich each part of size $a$ either with a $312$-avoiding indecomposable permutation from the symmetric group $\SSSS_{a}$, or, if $a=1$, with the symbol ${\bf D}$.
The symbol ${\bf D}$ stands for ``Down''
and it will be used for the part of the
bi-infinite sequence that is decreasing.
Note that if $a=1$, then the enrichment is either the identity permutation $1$ in $\SSSS_{1}$ or the symbol~${\bf D}$.  Let $E_{d}$ denote the set of all these enriched cyclic compositions.  Note that the number of enrichments of a part of size $a$ is the Catalan number $C_{a-1}$ plus the Kronecker delta $\delta_{a,1}$.  Hence, by Lemma~\ref{lemma_central_binomial}, we know that the total number of these structures is the central binomial coefficient~$\binom{2d}{d}$.

We now describe a bijection $\Phi : E_{d} \longrightarrow W_{d}$.
Let $\mP = (\mB_1,\mB_2,\ldots,\mB_k)$ 
be an enriched cyclic composition in $E_{d}$.  Recall that the $i$th block $\mB_{i}$ is the image of the interval $[p_{i},q_{i}]$ under the quotient map $\Zzz \longrightarrow \Zzz_{d}$.  If the enrichment on the part $\mB_{i}$ is a permutation, we view it as a permutation $\pi_{i}$ on the set $[p_{i},q_{i}]$.  Let~$\overline{\mB_{i}}$ be the set $\overline{\mB_{i}} = \bigcup_{j \in \Zzz} [p_{i} + j d, q_{i} + j d]$.  Note that $\overline{\mB_{1}}, \overline{\mB_{2}}, \ldots, \overline{\mB_{k}}$ is a partition of the integers~$\Zzz$.  Furthermore, extend $\pi_{i}$ by the relation $\pi_{i}(j+d) = \pi_{i}(j)+d$.  That is, now $\pi_{i}$ is a bijection on the set~$\overline{\mB_{i}}$.
Construct a bi-infinite sequence~$f$ by
$$     f(j)
=
\begin{cases}
              \exp(\pi_{i}(j)) & \text{ if } j \in \overline{\mB_{i}} 
                               \text{ and part $\mB_{i}$ is enriched with a permutation,} \\
             -\exp(j) & \text{ if } j \in \overline{\mB_{i}} 
                               \text{ and part $\mB_{i}$ is enriched with the symbol ${\bf D}$.}
                             \end{cases}
$$
By construction, the bi-infinite sequence $f$ is $d$-periodic.
It also is $312$-avoiding.
\begin{claim}
The bi-infinite sequence $f$ is $312$-avoiding.
\end{claim}
Assume not, that is, there are three integers $x<y<z$ such that $f(y) < f(z) < f(x)$.  Assume that $f(y) < 0 < f(z)$.  Then both $f(x)$ and $f(z)$ are positive.  Since $f(y)$ is negative, $x$ and $z$ belong to different intervals of the form $\mB_{i} + j \cdot d$.  But between these types of intervals $f$ is increasing, yielding the contradiction $f(x) < f(z)$.  Hence if $f(y)$ is negative then so is $f(z)$.  But the negative values of $f$ form a decreasing sequence, since this is a subsequence of $-\exp(j)$.  This contradicts $f(y) < f(z)$.  Now assume that $f(y) > 0$.  Since $f(x) > f(y)$, $x$ has to belong to the same interval $\mB_{i} + j \cdot d$ as $y$.  Similarly, since $f(x) > f(z)$, $z$~has to belong to the same interval $\mB_{i} + j \cdot d$ as $x$.  This contradicts to the fact that the block $\mB_{i}$ was enriched with a $312$-avoiding permutation,
proving the claim.

Finally, we construct an infinite walk $(\ldots,\sigma_{-1},\sigma_{0},\sigma_{1},\sigma_{2}, \ldots)$ in the graph $G(n,312)$ by letting $\sigma_{i}$ be the standardization $\Pi(f(i), f(i+1), \ldots, f(i+n))$.  Note that this permutation is $312$-avoiding and as an edge in the graph, it has the tail $\Pi(f(i), f(i+1), \ldots, f(i+n-1))$ and the head $\Pi(f(i+1), f(i+2), \ldots, f(i+n))$.  Since $f$ is $d$-periodic the infinite walk has period $d$.  Restricting the walk to $(\sigma_{1}, \sigma_{2}, \ldots, \sigma_{d})$ gives a closed walk in the set $W_{d}$.  This completes the description of the map $\Phi$.

\begin{example}
As an example of how the bijection $\Phi$ works, consider the cyclic composition of $d=8$ consisting of the interval $[-1,1]$ and the five singletons $\{2\}, \{3\}, \{4\}, \{5\}, \{6\}$.  Enrich the singletons $\{2\}$, $\{3\}$ and $\{5\}$ with the symbol ${\bf D}$. Enrich the interval $[-1,1]$ with the permutation $231$ and the two remaining singletons with the permutation $1$.  Then a graphical representation of the associated sequence $f$ is in Figure~\ref{figure_four}.  In the graph $G(8,312)$ the sequence $f$ describes the $8$-cycle:
$$ 45362178, 43521786, 34216758, 43267581,
32564718, 35647281, 56473821, 54637218 . $$ In the graph $G(9,312)$ the sequence $f$ describes the $8$-cycle whose three first edges are $453621897$, $435217869$ and $453278691$.
\label{example_ex}
\end{example}

\begin{figure}
\setlength{\unitlength}{0.4 mm}
\begin{center}
  \begin{picture}(150,200)(0,-80) \thicklines

\newcommand{\pppp}{\circle*{3.0}}

\qbezier(-5,15)(-2.5,27.5)(0,40)
\put(0,40){\pppp}
\qbezier(0,40)(5,45)(10,50)
\put(10,50){\pppp}
\qbezier(10,50)(15,40)(20,30)
\put(20,30){\pppp}
\qbezier(20,30)(25,45)(30,60)
\put(30,60){\pppp}
\qbezier(30,60)(35,20)(40,-20)
\put(40,-20){\pppp}
\qbezier(40,-20)(45,45)(50,70)
\put(50,70){\pppp}
\qbezier(50,70)(55,20)(60,-30)
\put(60,-30){\pppp}
\qbezier(60,-30)(65,-35)(70,-40)
\put(70,-40){\pppp}
\qbezier(70,-40)(75,25)(80,90)
\put(80,90){\pppp}
\qbezier(80,90)(85,95)(90,100)
\put(90,100){\pppp}
\qbezier(90,100)(95,90)(100,80)
\put(100,80){\pppp}
\qbezier(100,80)(105,95)(110,110)
\put(110,110){\pppp}
\qbezier(110,110)(115,30)(120,-50)
\put(120,-50){\pppp}
\qbezier(120,-50)(125,35)(130,120)
\put(130,120){\pppp}
\qbezier(130,120)(135,30)(140,-60)
\put(140,-60){\pppp}
\qbezier(140,-60)(145,-65)(150,-70)
\put(150,-70){\pppp}
\qbezier(150,-70)(152.5,-17.5)(155,35)

\put(-3,0){
\put(0,-80){\tiny -1}
\put(10,-80){\tiny 0}
\put(20,-80){\tiny 1}
\put(30,-80){\tiny 2}
\put(40,-80){\tiny 3}
\put(50,-80){\tiny 4}
\put(60,-80){\tiny 5}
\put(70,-80){\tiny 6}
\put(80,-80){\tiny 7}
\put(90,-80){\tiny 8}
\put(100,-80){\tiny 9}
\put(110,-80){\tiny 10}
\put(120,-80){\tiny 11}
\put(130,-80){\tiny 12}
\put(140,-80){\tiny 13}
\put(150,-80){\tiny 14}
}

\end{picture}
\end{center}
\caption{A schematic representation of the sequence appearing in Example~\ref{example_ex}.}
\label{figure_four}
\end{figure}
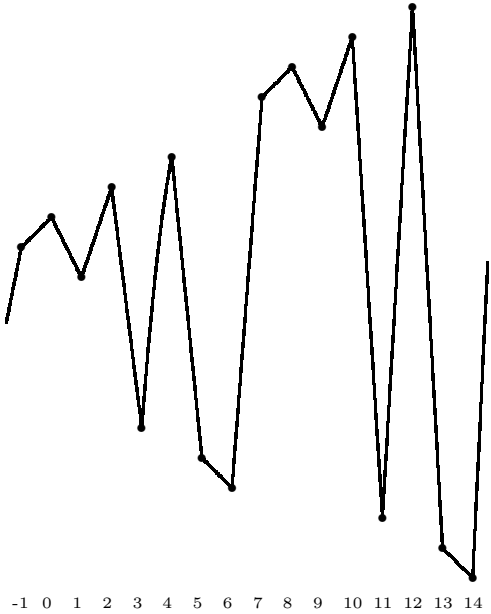

\section{The inverse bijection}
\label{sec-inv-bijection}

One can always lift an infinite walk in the graph to a bi-infinite sequence.  However, as Remark~\ref{remark_be_careful} shows, an infinite walk could lift to several non-equivalent sequences, and they do not all have the desired properties.  Thus, when lifting a walk to a sequence we have the additional requirements in
Conditions~\eqref{equation_extra_condition}
and~\eqref{equation_extra_condition_2}.
Their interpretation is that we do not introduce an inversion in the bi-infinite sequence, unless we are required to do so by a local condition.
\begin{remark}
{Consider the two bi-infinite sequences $$ h_{1}(n) = n + (-1)^{n} \:\:\:\: \text{ and } \:\:\:\: h_{2}(n) = n+2 (-1)^{n}.  $$ Observe that they both encode the same $2$-cycle in $G(2,312)$.  That is, $$ \Pi(h_{1}(n),h_{1}(n+1),h_{1}(n+2)) = \Pi(h_{2}(n),h_{2}(n+1),h_{2}(n+2)) = \begin{cases}
     132 & \text{if $n$ is odd,} \\
     213 & \text{if $n$ is even.}
   \end{cases}
$$
However, note that $h_{2}$ is not $312$-avoiding, whereas $h_{1}$ is.  Furthermore, $h_{2}$ does not have any cut points, whereas $h_{1}$ does.  Hence, when constructing the inverse map to $\Phi$ we have to be careful in constructing a bi-infinite sequence which is $312$-avoiding. 
}
\label{remark_be_careful}
\end{remark}

We now construct the inverse map of $\Phi$.  Given the closed walk $(\sigma_{1}, \sigma_{2}, \ldots, \sigma_{d})$ in $W_{d}$ we extend it to an infinite walk by letting $\sigma_{j+d} = \sigma_{j}$ for all integers $j$.

We are going to find a sequence $\ldots, g(-1), g(0), g(1), g(2), \ldots$ such that $\Pi(g(i),g(i+1), \ldots, g(i+n)) = \sigma_{i}$ for all integers $i$.  To find such a sequence, let $g(k) = \sigma_{1}(k)$ for $1 \leq k \leq n+1$.  Now alternate the following two steps to extend $g$ to all of the integers.
\begin{itemize}
\item[(+)] Assume that we have picked the values $g(i), g(i+1), \ldots, g(j-1)$ of the sequence.  We will now pick the value of $g(j)$.  That is, we are extending the sequence in the positive direction.  Let $\sigma$ be the permutation $\sigma_{i-n}$.  Let $a$ and $b$ be the real numbers (including $\pm \infty$) given by
  \begin{align*}
a & =
\begin{cases}
  g(\sigma^{-1}(\sigma(n+1)-1)+s) & \text{if } \sigma(n+1) > 1,\\
  -\infty & \text{if } \sigma(n+1) = 1,\\
\end{cases}  \\
b & =
\begin{cases}
  g(\sigma^{-1}(\sigma(n+1)+1)+s) & \text{if } \sigma(n+1) < n+1,\\
  \infty & \text{if } \sigma(n+1) = n+1,\\
\end{cases}
\end{align*}
where $s$ is the shift $s = j-n-1$.  Then any real number $x$ in the open interval $(a,b)$ satisfies $\Pi(g(j-n), \ldots, g(j-1), x) = \sigma$.  However, we have one more requirement, we will pick $x$ as large as possible with respect to the already picked values $g(i), g(i+1), \ldots, g(j-n-1)$.  That is, we pick $g(j) = x$ such that
\begin{equation}
  \max(a, \{g(k) \: : \: g(k) < b, i \leq k \leq j-1\}) < x < b .
  \label{equation_extra_condition}
\end{equation}
Note that we are picking a real number $g(j) = x$
in an open interval.
If the interval is bounded, it is fine to use
the average of the two endpoints.
However, any number in the open interval will do,
since there are plenty (read infinite) of real numbers in the interval
both less than and greater than our pick.

\item[(--)] Now we extend the sequence in the negative direction.  Assume that we have picked the values $g(i+1), g(i+2), \ldots, g(j)$ of the sequence.  We will now pick the value of $g(i)$.  Let $\sigma$ now denote the permutation $\sigma_{i}$.  Let the two bounds $a$ and $b$ be given by
  \begin{align*}
  a & =
  \begin{cases}
  g(\sigma^{-1}(\sigma(1)-1)+s) & \text{if } \sigma(1) > 1,\\
  -\infty & \text{if } \sigma(1) = 1,\\
\end{cases}  \\
b & =
\begin{cases}
  g(\sigma^{-1}(\sigma(1)+1)+s) & \text{if } \sigma(1) < n+1,\\
  \infty & \text{if } \sigma(1) = n+1,\\
\end{cases}
\end{align*}
where $s = i-1$.  Yet again, any real number $x$ in the open interval $(a,b)$ satisfies $\Pi(x, g(i+1), \ldots, g(i+n)) = \sigma$.  However, now we pick $x$ as small as possible with the already picked values $g(i+1), \ldots, g(j)$.  That is, we pick $g(i) = x$ such that
\begin{equation}
  a < x < \min(b, \{g(k) \: : \: g(k) > a, i+1 \leq k \leq j\})  
  \label{equation_extra_condition_2}
\end{equation}
\end{itemize}

The purpose of the two conditions
in~\eqref{equation_extra_condition}
and~\eqref{equation_extra_condition_2} is to avoid introducing any extra inversions in the sequence $g$.  These conditions will come into play at the end of this construction in the case when the set $D$ (to be defined soon) is empty.

\begin{claim}
  The sequence $g$ is locally $312$-avoiding, that is, if $i < j < k$, where $k-i \leq n$ then $\Pi(g(i),g(j),g(k)) \neq 312$.
  \label{claim_one}
\end{claim}
This holds true since $\sigma_{i}$ is $312$-avoiding.

\begin{claim}
For all $i$ we have $g(i) \neq g(i+d)$.
\end{claim}

Since $d \leq n$ and $\Pi(g(i), g(i+1), \ldots, g(i+d), \ldots, g(i+n))$ is a permutation, $g(i)$ and $g(i+d)$ are distinct.

\begin{claim}
For $i < j$ and $j-i < n$
the inequality
$g(i) < g(j)$
is equivalent to
$g(i+d) < g(j+d)$.
\label{claim_two}
\end{claim}
Since $\sigma_{i} = \sigma_{i+d}$
we have the string of the equivalences
$g(i) < g(j)
  \Longleftrightarrow 
\sigma_{i}(1) < \sigma_{i}(j-i+1)
  \Longleftrightarrow 
\sigma_{i+d}(1) < \sigma_{i+d}(j-i+1)
  \Longleftrightarrow 
g(i+d) < g(j+d)$.

Hence, the bi-infinite sequence $g$ decomposes into $d$ sequences, each of which is monotone.  We now partition the integers $\Zzz$ into two sets $D = \{i \: : \: g(i) > g(i+d)\}$ and $U = \{i \: : \: g(i) < g(i+d)\}$.  Note that since $d \leq n$ we have that $i \in D$ is equivalent to $i+d \in D$.  That is, $D$ consists of the sequences that are decreasing and $U$ of the increasing sequences.

\begin{claim}
The subsequence $\{g(i)\}_{i \in D}$ is decreasing.
\label{claim_three}
\end{claim}
Assume that
$\{g(i)\}_{i \in D}$
is not decreasing. Then there are two entries $i, j \in D$ such that $i < j$, $j-i \leq d-1$, and $g(i) < g(j)$.  Also, we have that $g(j-d) > g(j)$.
Combining the last two inequalities gives an
occurrence of $312$ since $\Pi(g(j-d), g(i), g(j)) = 312$.
This contradiction proves that
$\{g(i)\}_{i \in D}$
is decreasing.

\begin{claim}
  The values of the sequence $\{g(i)\}_{i \in D}$ are all smaller than the values of $\{g(j)\}_{j \in U}$.
  \label{claim_four}
\end{claim}
We begin when $i$ and $j$ are close to each other, that is, when $i < j < i+d$, $i \in D$ and $j \in U$.  Assume that $g(i) > g(j)$.  Then we have the string of inequalities $g(i-d) > g(i) > g(j) > g(j-d)$ implying that $\Pi(g(i-d), g(j-d), g(i)) = 312$, a contradiction.  Hence, we conclude that $g(i) < g(j)$.  Now pick $i^{\prime} \in D$ and $j^{\prime} \in U$.  If $i^{\prime} < j^{\prime}$ let $i = i^{\prime}$ and $j = j^{\prime} - d \cdot \lfloor (j^{\prime} - i^{\prime})/d \rfloor$.  If $i^{\prime} > j^{\prime}$ let $i = i^{\prime} - d \cdot \lceil (i^{\prime} - j^{\prime})/d \rceil$ and $j = j^{\prime}$.  (Here $\lfloor\cdot\rfloor$ and $\lceil\cdot\rceil$ are the usual floor and ceiling functions.)  In both cases we have $i \in D$, $j \in U$ and $i < j < i+d$.  Furthermore, we have that $g(i^{\prime}) \leq g(i) < g(j) \leq g(j^{\prime})$, proving the claim.

Now assume that $D$ is non-empty.  The case when $D$ is empty requires an extra argument, which we postpone to the end of this section.  Pick $p_{1}$ to be an element in the set $D$.  Decompose the interval $[p_{1},p_{1}+d-1]$ of cardinality $d$ into smaller intervals, according to the rules:
\begin{itemize}
\item[(d)] If $i \in D \cap [p_{1},p_{1}+d-1]$ then let the singleton $\{i\} = [i,i]$ be an interval in the decomposition.  Moreover, enrich this singleton with the symbol ${\bf D}$.
\item[(u)] If $i \leq j$, $i-1,j+1 \in D$ and $[i,j] \subseteq U \cap [p_{1},p_{1}+d-1]$ then we use the argument at the end of Section~\ref{section_compositions} to decompose the interval $[i,j]$ into smaller intervals, each enriched with an indecomposable $312$-avoiding permutation.  That is, we use the permutation $\Pi(g(i), g(i+1), \ldots, g(j))$ to decompose the interval.
\end{itemize}
Let the decomposition of the interval $[p_{1},p_{1}+d-1]$ be $\{[p_{1},q_{1}], [p_{2},q_{2}], \ldots, [p_{r},q_{r}]\}$, where $q_{i} + 1 = p_{i+1}$.  Extend this decomposition to a decomposition $\{[p_{i}, q_{i}]\}_{i \in \Zzz}$ of the integers $\Zzz$ by letting $p_{i+r} = p_{i} + d$ and $q_{i+r} = q_{i} + d$.  Note that under the quotient map $\Zzz \longrightarrow \Zzz_{d}$ we obtain a cyclic composition.

\begin{claim}
  If the intervals $[p_{i},q_{i}]$ and $[p_{k},q_{k}]$ are not enriched with the symbol ${\bf D}$, $i < k$, $x \in [p_{i},q_{i}]$ and $z \in [p_{k},q_{k}]$ then we have $g(x) < g(z)$.
\end{claim}
First assume that $z-x \leq d$.  If there is no interval $[p_{j},q_{j}]$ between these two intervals ($i < j < k$)
that
is enriched with the symbol ${\bf D}$ then the inequality follows by the decomposition into indecomposable permutations in part (u) above.  If there is an interval $[p_{j},q_{j}]$ in between which is enriched with the symbol ${\bf D}$, then consider the pattern $\Pi(g(x), g(p_{j}), g(z))$.  Since $[p_{j},q_{j}]$ is enriched by ${\bf D}$ we have that $g(p_{j}) < g(x)$ and $g(p_{j}) < g(z)$.  Hence, if $g(x) > g(z)$, we obtain the pattern~$312$, a contradiction.  Finally, if $z-x > d$, we obtain the inequality by using that $U$ consists of the increasing sequences.

The last claim states that we do not lose information if we view the permutation enriching the interval $[p_{i},q_{i}]$ as a bijection on this interval.  The resulting composition, viewed as a cyclic composition with its enrichment, is the inverse image of the map $\Phi$.

When the set $D$ is empty, we need to be more careful to show that the sequence $g$ has a cut point.  We will use an argument similar to that in the second proof of Corollary~\ref{corollary_there_is_a_cut point}.  However, there is an added complication since all we know is that $g$ is locally $312$-avoiding.  By Condition~\eqref{equation_extra_condition} we have the following lemma.
\begin{lemma}
  Assume that $j < k$ and $g(j) > g(k)$.  Then there is a chain $j = j_{0} < j_{1} < \cdots < j_{L} = k$ such that $g(j_{0}) > g(j_{1}) > \cdots > g(j_{L})$ and $j_{h+1} - j_{h} \leq n$ for all indices $0 \leq h \leq j_{L}-1$.
\label{lemma_chain}
\end{lemma}
\begin{proof}
  When we selected the value of $g(j_{L})$, we picked this value in an interval $(a,b)$ where $b$ was one of the values from the list $g(j_{L}-n), \ldots, g(j_{L}-2), g(j_{L}-1)$.  Hence, let $j_{L - 1}$ be the index such that $g(j_{L - 1}) = b$.

Assume that $g(j) < g(j_{L - 1})$.
Condition~\eqref{equation_extra_condition} states that we picked $g(j_{L})$ as large as possible in the interval $(a,b)$. Hence, the assumption $g(j) < g(j_{L - 1})$ implies that $g(j) < g(j_{L})$, a contradiction.  We conclude that $g(j) > g(j_{L - 1})$. By iterating this argument we obtain the chain.
\end{proof}

Now the argument is the same as in the second proof of Corollary~\ref{corollary_there_is_a_cut point}, except in the case when we use the $312$-avoidance.  That is, we have picked $(i,g(i))$ as a minimal element in the poset~$P$ and $(j,g(j))$ to be an element larger than or equal to the minimal element $(i,g(i))$ maximizing the second coordinate.  Observe that $i-j \leq d$.  The local $312$-avoidance condition implies that the poset order between $(i,g(i))$ and $(j,g(j))$ is a chain.  That is, we have the string of inequalities $g(j) > g(j+1) > \cdots > g(i-1) > g(i)$.

The remaining case is to show that there is no index $k$ such that $i < k$ and $g(i) < g(k) < g(j)$.  First pick $j^{\prime}$ in the interval $[j,i-1]$ such that $g(j^{\prime}+1) < g(k) < g(j^{\prime})$.  Next use Lemma~\ref{lemma_chain} to pick the first element of the chain $j_{1}^{\prime}$ such that $j^{\prime} < j_{1}^{\prime} \leq j^{\prime} + n$ and $g(j^{\prime}+1) < g(j_{1}^{\prime}) < g(j^{\prime})$.  However, this is a $312$-pattern, contradicting the assumption that there is such a $k$.  Hence, we conclude that the sequence $g$ has a cut point.

Let the cut point be $p_{1}-1$.  Consider the composition of the interval $[p_{1}, p_{1}+d-1]$ consisting of one part.  That is, this part is the interval~$[p_{1}, p_{1}+d-1]$.  Now, in a way similar to part (u) above, we decompose this interval into smaller intervals, each enriched with an indecomposable $312$-avoiding permutations using the permutation $\Pi(g(p_{1}), g(p_{1}+1), \ldots, g(p_{1}+d-1))$.  This completes the inverse of $\Phi$ in the case when the set $D$ is empty.

\section{Open problems}\label{sec-open}

In conclusion, we list a few open problems.

\begin{question}
  The sequence $\frac{1}{d} \sum_{e|d} \mu\left(\ffrac{d}{e}\right) \binom{2e}{e}$ from Theorem~\ref{theorem_number_of_cycles} appears in the On-Line Encyclopedia of Integer Sequences \cite{oeis} as sequence A060165.  It has been previously studied by Puri and Ward~\cite{Puri_Ward_I,Puri_Ward_II}.  When $q$ is a prime power, the number of monic irreducible polynomials of degree $d$ in $GF(q)[x]$ is given by
Equation~\eqref{equation_de_Bruijn}.
Is there a similar algebraic interpretation for the numbers occurring in Theorem~\ref{theorem_number_of_cycles}?
\end{question}

\begin{question}
  Can the number of $d$-cycles in the graph $G(n,321)$ be determined?  Equivalently, what is the number of closed walks in $G(n,321)$ of length $d$?  Of course, the same question can be asked for any set of patterns of length 3 or more, as well as for the entire graph $G(n)$.
\end{question}

\begin{question}
  One of the earliest results on De Bruijn graphs is that the number of complete cycles, also known as Eulerian cycles, is given by
$$    \frac{(q!)^{q^{n-1}}}{q^{n}}. $$
This result goes back to 1894 by Flye Sainte-Marie~\cite{Flye} in the case when $q=2$ and the general case to van Aardenne-Ehrenfest and De Bruijn~\cite{AT_B}.  Is there an analogous result for the graph $G(n)$ of overlapping permutations?  It seems, from the small examples we have studied, that the number of complete cycles in the graph of overlapping permutations has small prime factors, which gives hope that there is an explicit formula.
\end{question}

\begin{question}
Observe that $G(n,312)$ does not have any complete cycles since there are vertices with different out- and in-degrees.
Hence, it is natural to ask: For which sets $S$ of patterns do all vertices in the graph $G(n,S)$ have the same out- and in-degree? Also, what is the number of complete cycles in these graphs?
\end{question}

\begin{question}
  Lastly, it would be interesting to find bijective proofs of Lemmas~\ref{lemma_central_binomial} and~\ref{lemma_cyclic_composition_Catalan}.
\end{question}

\section*{Acknowledgments}

The authors thank the two referees for their
helpful comments.
The first author was partially supported by National Science Foundation grant DMS~0902063 and National Security Agency grant~H98230-13-1-0280.  The last author was supported by grant no.\ 090038013 from the Icelandic Research Fund.

\newcommand{\journal}[6]{{\sc #1}, {#2}, {\it #3}, {\bf #4} (#5), #6.}
\newcommand{\bookf}[5]{{\sc #1,} ``#2,'' #3, #4, #5.}

\end{document}